\documentclass[12pt]{amsart}

\usepackage{amsthm}
\usepackage{amsmath}
\usepackage{amsfonts,amssymb,latexsym}
\usepackage[english]{babel}
\usepackage{epsfig,epsf}
\usepackage{tikz}
\usepackage{enumerate}
\usepackage{color}
\newtheorem{thm}{Theorem}
\newtheorem{lem}{Lemma}

\newcommand\subsetsim{\mathrel{%
  \ooalign{\raise0.2ex\hbox{$\subset$}\cr\hidewidth\raise-0.9ex\hbox{\scalebox{0.9}{$\sim$}}\hidewidth\cr}}}

\newcommand{\Z}{{\mathbb Z}}

\title[A note on iterated sumsets races]{A note on iterated sumsets races}
\author{Paul P\'eringuey}

\address{Department of Mathematics, University of British Columbia, Vancouver, British Columbia, Canada\\ Pacific Institute for the Mathematical Sciences, Vancouver, British Columbia, Canada.}
\email{peringuey@math.ubc.ca}

\author{Anne de Roton}

\address{Universit\' e de Lorraine, Institut Elie Cartan de Lorraine, CNRS, Nancy, F-54000, France.}
\email{anne.de-roton@univ-lorraine.fr}

\begin{document}

\begin{abstract}
This short note answers a question raised by Nathanson \cite{Nath25} about "races" between iterated sumsets. We prove that for any integer $n$, there are finite sets of integers $A$ and $B$ with same diameter such that the signs of the elements of the sequence $(|hA|-|hB|)_h$ changes at least $n$ times. Kravitz proved in \cite{Kravitz} a much better result. This brief and modest note may serve as a stepping stone towards his work.
\textbf{Mathematic Subject Classification:} 11P70.
\end{abstract}

\maketitle

\section{Introduction}

Given a nonempty subset $A$ of integers and a positive integer $h$, the $h$-fold sumset of the set $A$, denoted by $hA$ is the set of all sums of $h$ elements of $A$ 
$$hA=\{a_1+\cdots+a_h,\,{ with }\,a_i\in A \mbox{ for all }i\in [1,h]\}.$$
The size (or cardinality) of $hA$ will be the number of elements in $hA$ and denoted by $|hA|$.\\
In \cite{Nath25} Nathanson investigates inverse problems for the sequence of sumset sizes $(|hA|)_{h\ge 1}$ and especially the relation between $(|hA|)_{h\ge 1}$ and $(|hB|)_{h\ge 1}$ for sets $A$ and $B$ of the same size. 
He raises the following question.\\
{\bf Problem 2 Nathanson}\\
{\it For every integer $m\ge 3$ do there exist finite sets $A$ and $B$ of integers and an increasing sequence of positive integers $h_1<h_2<\cdots<h_m$ such that 
$$|h_i A|>|h_iB| \mbox{ if $i$ is odd}$$ 
and
$$|h_i A|<|h_iB| \mbox{ if $i$ is even.}$$ 
Do there exist such sets with $|A|=|B|$? Do there exist such sets with $|A|=|B|$ and ${\rm diam}(A)={\rm diam}(B)$ ?}\\
\\
In other words, Nathanson investigated "races" between sumsets, a problem that is part of a broader collection of questions he posed in a series of papers, including \cite{PANT6}.\\
In \cite{Nath25}, Nathanson gives an example of sets $A$ and $B$ such that $|A|>|B|$, ${\rm diam}(A)>{\rm diam}(B)$ and there exist integers $h_2$ and $h_3$ such that $|hA|<|hB|$ for $2\le h\le h_2$ and $|hA|>|hB|$ for $h\ge h_3$. This gives a partial answer to Problem 2 for $m=3$. \\ 
\\
In this paper, we will solve Nathanson's Problem 2.
\begin{thm}\label{main}
Let $m\ge 3$ be an integer. There exist finite sets $A$ and $B$ of integers and an increasing sequence of positive integers $h_1<h_2<\cdots<h_m$ such that 
$|A|=|B|$, ${\rm diam}(A)={\rm diam}(B)$ 
and
$$|h_i A|>|h_iB| \mbox{ if $i$ is odd}$$ 
and
$$|h_i A|<|h_iB| \mbox{ if $i$ is even.}$$ 
\end{thm} 

After writing this short note, we discovered that Noah Kravitz \cite{Kravitz} had, in January 2025, posted a comprehensive solution to Nathanson's Problem 2. His work addresses a more general and complex version of the problem, as he considers races between $k$ sets, rather than just two as in our study. Most notably, he proves that the sequence $(h_i)_i$ can be prescribed. Moreover, Kravitz extends his investigation beyond the integers, exploring sumsets in arbitrary infinite abelian groups.\\
In \cite{FKZ} Jacob Fox, Noah Kravitz and Shengtong Zhang go even further. They show that the values of $|hA|-|hB|$ can be prescribed and investigate the question of finding the smallest cardinality and diameter of sets which realize this prescribed values.\\
From every perspective, Kravitz’s results and his work with Fox and Zhang go well beyond what we present here. Nonetheless, we hope this brief and modest note may serve as a stepping stone, in a simpler setting, toward a better understanding of Nathanson's problems and Kravitz’s solution.

\section{Preliminary results}
While the growth of the first iterated sumsets can be rather erratic, we know that its ultimate growth is linear with a growth rate depending on the diameter. The following lemma due to Nathanson \cite{Nathanson1972} provides a description of this phenomenon. 
\begin{lem}[Nathanson]\label{lem_nath}
Let $A$ be a finite subset of integers with $\min(A)=0$, $\max(A)=N$ and $gcd(A)=1$. 
There exists some integer $h_0$, integers $b,c\ge 0$ and sets $A_1\subset [0,b-2]$, $A_2\subset [0,c-2]$ such that
$$\forall h\ge h_0, \quad hA=A_1\cup [b,hN-c]\cup(hN-A_2).$$ 
\end{lem}
This lemma implies that for $h\ge h_0$, $|hA|=hN-(b-|A_1|+c-|A_2|)+1$. This ultimately implies that the set with the largest diameter will prevail in the race, this will be crucial for our first construction.\\
\\
The following lemma shows that, under suitable conditions, the growth of the first iterated sumsets of a subset $\cup_{j\in I}(j\tau+A)$ of $\Z$ mirrors the growth of the iterated sumsets of the subset $I\times A$ of $\Z^2$. 
\begin{lem}\label{small_h}
Let $A$ be a finite subset of integers with $\min(A)=0$, $\max(A)=N$. 
Let $I$ be a finite subset of integers and $\tau>N$ be an integer.\\ 
For any positive integer $h<\tau/N$, the set $A'=\cup_{j\in I}(j\tau+A)$ satisfies $|hA'|=|hI|\cdot |hA|$.
\end{lem}
\begin{proof}
If $A'=\cup_{j\in I}(j\tau+A)$ then 
$hA'=\cup_{j\in hI}(j\tau+hA)$ and the sets $j\tau+hA$ are distinct as long as $hN<\tau$. This implies
 $$|hA'|=\sum_{j\in hI}|j\tau+hA|=\sum_{j\in hI}|hA|=|hI|\cdot|hA|.$$
\end{proof}

The following Lemma describes the ultimate growth of the iterated sumsets of a union of translates of a subset $A$ of $\Z$. 
\begin{lem}\label{h_large}
Let $A$ be a finite subset of integers with $\min(A)=0$, $\max(A)=N$ and $gcd(A)=1$. 
\\There exist some integers $\delta\in[0,N+1]$, $\gamma\ge 0$ and $h_0\ge 1$ such that
for all positive integer $r\ge 2$ and $\tau\ge\gamma$, the set $A'=\cup_{j=0}^{r-1}(j\tau+A)$ satisfies
for all $h\ge \max\left(h_0,\frac{\gamma+\tau-1}N\right)$
$$|hA'|=h((r-1)\tau+N)+1-\delta$$  and $hA'$ contains an interval of length $h(N+(r-1)\tau)-\gamma+1.$

\end{lem}

\begin{proof}
According to Lemma \ref{lem_nath}, there exists some integer $h_0$ such that for $h\ge h_0$,\\
$hA=A_1\cup [b,hN-c]\cup(hN-A_2)$ with $A_1\subset [0,b-2]$, $A_2\subset [0,c-2]$ and some integers $b,c\ge 0$.\\
For $h\ge h_0$, we have 
\begin{align*}
hA'&=\cup_{j=0}^{h(r-1)}(j\tau+hA)=\cup_{j=0}^{h(r-1)}\left(j\tau+(A_1\cup [b,hN-c]\cup(hN-A_2))\right)
\end{align*}
If $b+\tau\le hN-c+1$ then $\left(j\tau+[b,hN-c]\right)\cup\left((j+1)\tau+[b,hN-c]\right)=j\tau+[b,\tau+hN-c]$ is an interval and
\begin{align*}
hA'&=A_1\cup [b,hN-c+h(r-1)\tau]\cup(h(r-1)\tau+hN-A_2)
\end{align*}
thus for $\tau\ge \max(b,c)$
$$|hA'|= |A_1|+|A_2|+h(N+(r-1)\tau)-(b+c)+1.$$
Taking $\delta= b-|A_1|+c-|A_2|$ and $\gamma=b+c$, we get the announced result.
\end{proof}

\section{Theorem \ref{main} without the condition on the diameter.}
We begin with a partial result, Theorem \ref{main} without the condition that $A$ and $B$ have the same diameter.

\begin{thm}\label{sans_diam}
Let $m\ge 3$ be an integer. There exist finite sets $A$ and $B$ of integers and an increasing sequence of positive integers $h_1<h_2<\cdots<h_m$ such that 
$|A|=|B|$
and
$$|h_i A|>|h_iB| \mbox{ if $i$ is odd}$$ 
and
$$|h_i A|<|h_iB| \mbox{ if $i$ is even.}$$ 
\end{thm} 

\begin{proof}
We proceed by induction on the integer $m$. Given two sets $A$ and $B$ that satisfy the theorem for a sequence $h_1,\cdots,h_m$, we construct new sets $A'$ and $B'$ that satisfy the theorem for the extended sequence $h_1,\cdots,h_m, h_{m+1}$ with $h_{m+1}>h_m$. 
The key idea is to define $A'$ and $B'$ as boxes containing appropriate ranges of $A$ and $B$ respectively (as in Lemma \ref{h_large}). This allows us to preserve the relative order of the sizes $|hA|$ and $|hB|$ for small values of $h$ thanks to Lemma \ref{small_h}, while adjusting the dimension of the boxes so as to determine the desired ultimate winner of the race. 

Let $A,B$ be some sets of integers such that $|A|=|B|$, $\min(A)=\min(B)=0$, $\max(\max(A),\max(B))=N$ and let $h_1<h_2<\cdots<h_m$ be an increasing sequence of $m$ positive integers such that 
$$|h_i A|>|h_iB| \mbox{ if $i$ is odd}$$ 
and
$$|h_i A|<|h_iB| \mbox{ if $i$ is even.}$$ 

Without loss of generality, we may assume that $gcd(A)=gcd(B)=1$, otherwise we replace each set with an appropriately scaled version.\\
Let $\alpha, \beta>h_m N$ be two large integers and $r\ge 2$ be a positive integer
and define
$$A'=\cup_{j=0}^{r-1}(j\alpha+A), \quad B'=\cup_{j=0}^{r-1}(j\beta+B).$$
We have $\min(A')=\min(B')=0$, $|A'|=|B'|=r|A|=r|B|$ and $\max(A')=(r-1)\alpha+\max(A)$, $\max(B')=(r-1)\beta+\max(B)$.

If $h<\min(\alpha,\beta)/N$, by Lemma \ref{small_h}, 
$$\begin{cases}|hA'|= |h[0,r-1]|\cdot |hA|=(h(r-1)+1) |hA|\\|hB'|= (h(r-1)+1) |hB|\end{cases}$$
thus for $i\le m$,
$|h_i A'|>|h_iB'| \mbox{ if $i$ is odd}$
and
$|h_i A'|<|h_iB'| \mbox{ if $i$ is even}$.\\
Take $\alpha,\beta>h_mN$ large enough and 
$$\begin{cases} \alpha<\beta &\mbox{ if $m$ is odd}\\ \alpha>\beta &\mbox{ if  $m$ is even.}\end{cases}$$
By Lemma \ref{h_large}, for $h$ large enough
$$|hA'|=h((r-1)\alpha+\max(A))+1-\delta_A \quad \mbox{ and }\quad |hB'|=h((r-1)\beta+\max(B))+1-\delta_B$$
with $0\le\delta_A\le \max(A)$, $0\le \delta_B\le \max(B)$.\\
If $m$ is odd, then for $h$ large enough and $r\ge 1+\frac{N-1}{|\beta-\alpha|}$
$$|hA'|\le h((r-1)\alpha+N)+1< h(r-1)2\beta\le |hB'|$$
whereas if $m$ is even, then $|hA'|>|hB'|$.\\
Therefore, there exist $h_{m+1}>h_m$ such that $|h_{m+1}A'|<|h_{m+1}B'|$ if $m$ odd, i.e. $m+1$ even and $|h_{m+1}A'|>|h_{m+1}B'|$ if $m$ even, i.e. $m+1$ odd.\\
Finaly we found two sets $A'$ and $B'$ such that $|A'|=|B'|$ and an increasing sequence $h_1<h_2<\cdots<h_m<h_{m+1}$ such that 
$$|h_i A|>|h_iB| \mbox{ if $i$ is odd}$$ 
and
$$|h_i A|<|h_iB| \mbox{ if $i$ is even.}$$ 
\end{proof}

\section{Proof of Theorem \ref{main}}
To prove Theorem \ref{main}, we once again proceed by induction on $m$, employing similar arguments but using different index sets $I$ and $J$ in place of $[0,r-1]$.
The central idea remains to construct unions of translates of the sets $A$ and $B$ in a way that preserves the order of the initial sumset sizes, while allowing us to control the eventual outcome of the race over a greater number of iterations.
To achieve this, we select index sets $I$ and $J$ of same diameter such that $|hI|=|hJ|$ for small $h$ and $|hI|<|hJ|$ for some larger $h$.\\
Let $A,B$ be some sets of integers such that $|A|=|B|$, $\min(A)=\min(B)=0$, $\max(A)=\max(B)=N$ and let $h_1<h_2<\cdots<h_m$ be an increasing sequence of $m$ positive integers such that 
$$|h_i A|>|h_iB| \mbox{ if $i$ is odd}$$ 
and
$$|h_i A|<|h_iB| \mbox{ if $i$ is even.}$$ 
By Nathanson's result Lemma \ref{lem_nath}, there exist some integers $a,b$ and $h_0$ such that\\ for $h\ge h_0,$ 
$\begin{cases}
|hA|=hN-a,\\|hB|=hN-b.
\end{cases} $\\
Let $H>\max(h_m,h_0,2)$ be a large integer. \\ 
Define $I=[0,2]\cup\{2H\}$ and $J=[0,1]\cup[2H-1,2H]$.\\
We have $hI=\bigcup_{i=0}^h\left(2Hi+[0,2(h-i)]\right)$ and $hJ=\bigcup_{i=0}^h\left(2Hi+[-i,(h-i)]\right)$, thus
$$|hI|=\begin{cases}(h+1)^2&\mbox{ if }h<H \\ H^2+2H(h-H+1)&\mbox{ if }h\ge H\end{cases}$$
and
$$|hJ|=(h+1)^2 \quad \mbox{ if }h<2H-1. $$
Take $h_{m+1}=2H-2$. We have
$|h_{m+1}J|=(2H-1)^2$ and $|h_{m+1}I|=3H^2-2H$, whereas $|hI|=|hJ|=(h+1)^2$ for $h<H$ thus for any $h_i$ with $i\le m$.

We are now ready to define our new sets $A'$ and $B'$.\\
Take $\tau$ a large integer such that $\tau>2HN$.\\
If $m$ is odd, thus $|h_mA|>|h_mB|$, define
$$A'=\cup_{j\in I}(j\tau+A), \quad B'=\cup_{j \in J}(j\tau+B).$$
If  $m$ is even, thus $|h_mA|<|h_mB|$, define
$$A'=\cup_{j\in J}(j\tau+A), \quad B'=\cup_{j \in I}(j\tau+B).$$
We will assume $m$ odd, the other case being similar.\\
For $i\le m$, we have $|h_iI|=|h_iJ|$ thus by Lemma \ref{small_h}, 
$$\begin{cases}
|h_i A'| =|h_iI||h_iA|>|h_iJ||h_iB|=|h_iB'| \mbox{ if $i$ is odd}\\
|h_i A'|<|h_iB'| \mbox{ if $i$ is even.}
\end{cases}$$
Since $h_{m+1}<2H<\tau/N$ and $2H-2\ge H\ge h_0$ we also have
\begin{align*}
|h_{m+1}B'|-|h_{m+1}A'|&=|h_{m+1}J||h_{m+1}B|-|h_{m+1}I||h_{m+1}A|\\
&=(2H-1)^2((2H-2)N-b)-(3H^2-2H)((2H-2)N-a)\\
&=(H^2-2H+1)(2H-2)N-b(4H^2-4H+1)+a(3H^2-2H)\\
&=2N(H-1)^3-b(4H^2-4H+1)+a(3H^2-2H)\\
\end{align*}
 
 Taking $H$ so large that $2N(H-1)^3-b(4H^2-4H+1)+a(3H^2-2H)>0$, we get $|h_{m+1}B'|>|h_{m+1}A'|$.
 
 Since 
 $$\begin{cases}{\rm diam}(A')=2H\tau+{\rm diam}(A)=2H\tau+{\rm diam}(B)={\rm diam}(B')\\|A'|=|I||A|=|J||B|=|B'|\end{cases}$$ we proved that there exist two sets $A'$ and $B'$ such that $|A'|=|B'|$, ${\rm diam}(A')+{\rm diam}(B')$ and  there exists an increasing sequence $h_1<h_2<\cdots<h_m<h_{m+1}$ such that 
$$|h_i A|>|h_iB| \mbox{ if $i$ is odd}$$ 
and
$$|h_i A|<|h_iB| \mbox{ if $i$ is even.}$$

 \vspace{3cm}

\bibliographystyle{alpha}
\bibliography{iterated_sums_races.bib}

\end{document}